\documentclass{amsart}
\usepackage{titlecaps}
\usepackage{amsfonts}
\usepackage{amsmath}
\usepackage{amssymb}
\usepackage{cite}
\usepackage{mathrsfs}
\usepackage{enumitem}
\usepackage{xcolor}
\usepackage{hyperref}
\usepackage{graphicx}

\newtheorem{theorem}{Theorem}[section]
\newtheorem{lemma}[theorem]{Lemma}
\newtheorem{corollary}[theorem]{Corollary}

\theoremstyle{definition}

\newtheorem{remark}[theorem]{Remark}
\theoremstyle{remark}

\numberwithin{equation}{section}

\begin{document}
	
	\title[Completeness of All Translates of a Function in the Orlicz Spaces]
	{A Note on the Completeness of All Translates of a Function in the Orlicz Spaces}
	
	\author[Bhawna Dharra]{Bhawna Dharra}
	\address{
		Bhawna Dharra:
		\endgraf
		Department of Mathematics
		\endgraf
		Indian Institute of Technology, Delhi, Hauz Khas
		\endgraf
		New Delhi-110016 
		\endgraf
		India
	}
	\email{bhawna.dharra@gmail.com}
	
	\author[S. Sivananthan]{S. Sivananthan}
	\address{
		S. Sivananthan:
		\endgraf
		Department of Mathematics
		\endgraf
		Indian Institute of Technology, Delhi, Hauz Khas
		\endgraf
		New Delhi-110016 
		\endgraf
		India
	}
	\email{siva@maths.iitd.ac.in}

	\date{}
	
	\begin{abstract}
		We give a characterization of those functions whose all translates are complete in certain Orlicz space $L^{\Phi}(\mathbb{R})$. As a consequence, we identified those discrete sets $\Lambda \subseteq \mathbb{R}$ such that there exists a function in $L^{\Phi}(\mathbb{R})$ whose $\Lambda$-translates are complete. We then prove the completeness of all translates of any simple step function in other Orlicz spaces. 
	\end{abstract}
	\maketitle

	
	\section{Introduction}
	
	Orlicz spaces offer a natural generalization of the classical Lebesgue spaces. They are formed using Orlicz functions, which are an extension of the properties of the function $x^p$ corresponding to the $L^p$-spaces. The extensiveness of Orlicz spaces gives rise to their several significant properties yielding fundamental applications in Fourier analysis, Stochastic analysis, Sobolev space embeddings, non-linear PDE and more \cite{raorenapplicationbook}.
	We are concerned with the completeness of the systems which are generated using all translates of a single function in the Orlicz space, i.e., for $f \in L^{\Phi}(\mathbb{R})$, when does 
	\begin{equation*}
		\overline{	\mbox{span }\{ \tau_{\lambda}f(\cdot):=f(\cdot - \lambda )\}_{\lambda \in \mathbb{R}} }~~=~~L^{\Phi}(\mathbb{R})?\\
	\end{equation*}
	For a survey of results on the completeness of the system of translates in the classical Lebesgue spaces, we refer to \cite[Chapter 11,12]{olevskiiulanovskii18}. 
	
	The classical Wiener-Tauberian theorem \cite{wiener32} characterizes such functions for the special case of $L^1(\mathbb{R})$ and $L^2(\mathbb{R})$.  It is proved that all translates of a function $f$ are complete in $L^1(\mathbb{R})$ (and in $L^2(\mathbb{R})$) precisely when its Fourier transform $\hat{f}$ does not vanish anywhere (and does not vanish $a.e.$ respectively). However, characterization of such functions in terms of the zero set of their Fourier transform is not always possible, as shown by Lev and Olevskii \cite{levolevskii11} for the spaces $L^p(\mathbb{R})$ when $1<p<2$. 
	
	Since the Orlicz spaces offer a bigger class of function spaces than the Lebesgue spaces, one can ask whether there are Orlicz spaces other than $L^1(\mathbb{R})$ and $L^2(\mathbb{R})$ for which we can characterize the functions whose all translates are complete in terms of the zero set of the Fourier transform. We proved that for Orlicz spaces continuously embedded in the space $L^1(\mathbb{R})$, such a characterization exists. As a result, we are able to characterize those discrete sets $\Lambda \subset \mathbb{R}$ such that the system of $\Lambda$-translates of some function is complete in these Orlicz spaces. The spaces for which such a characterization is known are $L^1(\mathbb{R})$ \cite{brunaolevskii2006}, some weighted $L^1_w(\mathbb{R})$ \cite{blank2006} spaces and the Hardy space $H^1(\mathbb{R})$ \cite{bhawna22} only.  
	
	Moreover, rather than looking for characterizing of these functions, one can also ask whether all translates of a particular function are able to span a dense set for the whole space or not. Agnew \cite{agnew1945I} proved that whenever you take any simple step function of the form $\chi_{[0,h]}$ for some $h>0$, then definitely its all translates are complete in $L^p(\mathbb{R})$ for every $p>1$. However, Wiener's theorem ensures that it is not true for $p=1$.  Inspired by the variation in the above result in the cases when $p=1$ and $p>1$, we look for conditions on the Orlicz function $\Phi$ so that all translates of any simple step function are complete in $L^{\Phi}(\mathbb{R})$ or not.

	
	\section{Preliminaries}\label{Preliminaries}
	
	An \textit{Orlicz function} is defined to be a convex function $\Phi:[0,\infty)\to [0,\infty]$ such that $\Phi(0)=0$ and $\lim\limits_{x \to \infty}\Phi(x)=\infty$. For a given Orlicz function $\Phi$ and a measurable set $\Omega \subseteq \mathbb{R}$, the \textit{Orlicz space} $L^{\Phi}(\Omega)$ is a Banach space defined as
	$$L^{\Phi}(\Omega):=\left\{f: f \text{\small{ is measurable on }}\Omega, \int\limits_{\Omega} \Phi\left(\frac{|f(x)|}{a}\right)dx < \infty \text{\small{ for some }}a>0\right\},$$
	equipped with the \textit{Luxemburg norm}
	$$\|f\|_{L^{\Phi}(\Omega)}=\inf\left\{ a > 0 ~:~ \int\limits_{\Omega} \Phi\left(\frac{|f(x)|}{a}\right)dx \leq 1 \right\},$$
	where we identify functions which are equal $a.e.$ Note that the Orlicz spaces corresponding to $\Phi(x)=x^p$ for $1 \leq p < \infty$ gives the classical Lebesgue spaces. Also,
	$L^1(\mathbb{R})*L^{\Phi_1}(\mathbb{R}) \subseteq L^{\Phi_1}(\mathbb{R}) $ for any Orlicz function $\Phi_1$. 
	
	A \textit{complementary Orlicz function to} a given Orlicz function $\Phi$ is defined by 
	$$\Psi(y)=\sup \{xy-\Phi(x)~:~x \geq 0\}, ~~y \geq 0.$$
	
	We have a generalized Young's inequality for a pair of complementary Orlicz functions $(\Phi,\Psi)$,
	$$xy \leq \Phi(x)+\Psi(y), ~~x,y \geq 0,$$
	where the equality is achieved if and only if either $y=\Phi^{\prime}(x)$ or $x=\Psi^{\prime}(y)$. As a consequence, we have the \textit{H\"older's inequality}: If $f \in L^{\Phi}(\Omega)$ and $g \in L^{\Psi}(\Omega)$, then $fg \in L^1(\Omega)$ and
	$$\int\limits_{\Omega}|fg|dx \leq 2 \|f\|_{L^{\Phi}(\Omega)}\|g\|_{L^{\Psi}(\Omega)}.$$
	With its help, a norm equivalent to $\|\cdot\|_{L^{\Phi}(\Omega)}$ can be defined on the Orlicz space $L^{\Phi}(\Omega)$, called the Orlicz norm
	$$\|f\|_{L^{(\Phi)}(\Omega)}:=\sup \left( \int\limits_{\Omega}|fg|dx ~~:~~ g \in L^{\Psi}(\Omega)\text{ with }\int\limits_{\Omega}\Psi(|g|)dx \leq 1 \right).$$
	The following equivalence relation between the Orlicz and the Luxemburg norm holds
	$$\|f\|_{L^{\Phi}(\Omega)} \leq \|f\|_{L^{(\Phi)}(\Omega)} \leq 2\|f\|_{L^{\Phi}(\Omega)} ~~ \text{ for all } f \in {L^{\Phi}(\Omega)}.$$

	We will work with a special yet wide class of Orlicz functions. We say $\Phi \in \Delta_2$ if there exists a constant $D>0$ such that 
	$$\Phi(2x) \leq D\Phi(x) \text{ for all }x \geq 0.$$

	The dual space $(L^{\Phi})^{*}$ of the Orlicz space $L^{\Phi}$ has the following characterization whenever $\Phi \in \Delta_2$, 
	\begin{theorem}\label{duality_theorem}
		Let $(\Phi,\Psi)$ be a pair of complementary Orlicz functions such that $\Phi \in \Delta_2$. Then $(L^{\Phi})^{*}=L^{\Psi}$ and for each $f^{*} \in (L^{\Phi})^{*}$, there exists a unique $g \in L^{\Psi}$ such that 
		$$f^{*}(f) = \int fg dx ,~~f \in L^{\Phi}.$$
		In particular, the Orlicz space $L^{\Phi}(\Omega)$ is reflexive if and only if $\Phi , \Psi \in \Delta_2$.
	\end{theorem}
	
	Also, the $\Delta_2$ condition also depicts the separability of the Orlicz space.
	\begin{theorem}\label{separability}
		For an Orlicz function $\Phi$, the space $L^{\Phi}(\Omega)$ is separable if and only if $\Phi \in \Delta_2$.
	\end{theorem}
	
	The intersection of two Orlicz spaces is again an Orlicz space:
	\begin{theorem}\label{intersection_orlicz_function} \cite[Proposition ~17.1.1]{book-foundations-of-symmetric-spaces}
		Let $\Phi_1, \Phi_2$ be two Orlicz functions. Define the function
		$$\Phi_1 \vee \Phi_2 := \max (\Phi_1, \Phi_2).$$
		It is again an Orlicz function and 
		$$L^{\Phi_1 \vee \Phi_2}(\Omega) = L^{\Phi_1}(\Omega) \cap L^{\Phi_2}(\Omega).$$
	\end{theorem}
	
	We can define some partial order for the class of Orlicz functions. For given two Orlicz functions $\Phi_1$ and $\Phi_2$, we have the following.
	\begin{enumerate}
		\item \cite[p.~207]{book-foundations-of-symmetric-spaces}
		We say that $\Phi_1$ \textbf{majorizes} $\Phi_2$ at $0$, denoted by ${\Phi_2 \prec_0 \Phi_1}$, if there exist $a, b, x_0 >0$ such that 
		$$\Phi_2(x) \leq b \Phi_1(ax) \text{ for all } 0\leq x \leq x_0.$$
		\item We say that $\Phi_1$ \textbf{ majorizes} $\Phi_2$ at $\infty$, denoted by ${\Phi_2  \prec_{\infty} \Phi_1}$, if there exist $a, b, x_0 >0$ such that 
		$$\Phi_2(x) \leq b \Phi_1(ax) \text{ for all } x \geq x_0.$$
		\item  We say that $\Phi_1$ \textbf{ majorizes} $\Phi_2$, denoted by ${\Phi_2 \prec \Phi_1}$, if there exist $a, b >0$ such that 
		$$\Phi_2(x) \leq b \Phi_1(ax) \text{ for all } x \geq 0.$$
	\end{enumerate}

	We have the following embedding of Orlicz spaces.
	\begin{theorem}\label{embedding_result}
		Let $(\Omega, \mu)$ be a $\sigma$-finite measure space and $\Phi_1$ and $\Phi_2$ be two Orlicz functions. 
		\begin{enumerate}
			\item If $\mu$ is non-atomic and $\mu(\Omega)=\infty$, then 
			$$L^{\Phi_1}(\Omega,\mu) \subseteq L^{\Phi_2}(\Omega, \mu) \Leftrightarrow \Phi_2 \prec \Phi_1.$$
			\item If $\mu$ is atomic and $\mu(\omega)>0$ for all $\omega \in \Omega$, then 
			$$L^{\Phi_1}(\Omega,\mu) \subseteq L^{\Phi_2}(\Omega, \mu) \Leftrightarrow \Phi_2 \prec_0 \Phi_1.$$
			\item If $\mu$ is non-atomic and $\mu(\Omega)<\infty$, then 
			$$L^{\Phi_1}(\Omega,\mu) \subseteq L^{\Phi_2}(\Omega, \mu) \Leftrightarrow \Phi_2 \prec_{\infty} \Phi_1.$$
		\end{enumerate}
	\end{theorem}
	We say that two Orlicz functions $\Phi_1$ and $\Phi_2$ are equivalent, denoted by $\Phi_1 \sim \Phi_2$, if $\Phi_1 \prec \Phi_2$ and $\Phi_2 \prec \Phi_1$. This implies that $L^{\Phi_1}(\mathbb{R})=L^{\Phi_2}(\mathbb{R})$ and the norms $\| \cdot \|_{L^{\Phi_1}}$ and $\| \cdot \|_{L^{\Phi_2}}$ are equivalent.

	For more details about the Orlicz spaces, we refer the interested reader to \cite{raorenapplicationbook, book-foundations-of-symmetric-spaces}.\\
	
	We will need the following inclusion result proved in \cite{bhawna_orlicz}, according to which the existence of a function whose $\Lambda$-translates are complete will ensure the same in another Orlicz space based on a comparison around $0$.
	\begin{theorem}\label{inclusion_result}\cite[Theorem 1.3]{bhawna_orlicz}
		Let $\Lambda \subseteq \mathbb{R}$ and $\Phi_0 \in \Delta_2$. Suppose there exists an $f \in L^{\Phi_0}(\mathbb{R})$ whose $\Lambda$-translates are complete in $L^{\Phi_0}(\mathbb{R})$. Then for every Orlicz function $\Phi \in \Delta_2$ such that $\Phi \prec_0 \Phi_0$, there exists $g \in L^{\Phi}(\mathbb{R})$ whose $\Lambda$-translates are complete in $L^{\Phi}(\mathbb{R})$.
	\end{theorem}
	
	We also need a notion of a certain density of a set. For any discrete set $\Lambda \subset \mathbb{R}$, its Beurling-Malliavin Density $D_{BM}(\Lambda)$ is defined as 
	\begin{equation*}
		D_{BM}(\Lambda):=\sup \left\{ D>0 ~:~ \substack{\mbox{ there is a substantial sequence of intervals }\{I_k\}_{k \in \mathbb{N}} \\ \mbox{ such that } \#(\Lambda \cap I_k) \geq D|I_k|, \mbox{ for all } k\in \mathbb{N}} \right\}.
	\end{equation*}
	We say that a sequence of intervals $I_k, k \in \mathbb{N}$ is substantial if it belongs to the positive or negative half-axis, $|I_k|>1$ for all $ k \in \mathbb{N}$, the intervals are disjoint, and
	$$\sum\limits_{k=1}^{\infty}\left( \frac{|I_k|}{\text{dist}(I_k,0)}  \right)^2 = \infty.$$
	

	\section{Main Results}
	
	First, we characterize those functions in Orlicz spaces, which are continuously embedded in $L^1(\mathbb{R})$, whose all translates are complete in terms of the zero set of their Fourier transform. The main tool used in the proof is from the Gelfand theory of the commutative Banach algebras.\\

	Observe that $L^{\Phi}(\mathbb{R})$ is a Banach algebra under convolution and the subspace generated by all translates of any function is a closed ideal of $L^{\Phi}(\mathbb{R})$. In fact, more is true.
	\begin{lemma}\label{translation_invariant_ideal}
		Every closed translation-invariant subspace of $L^{\Phi}(\mathbb{R})$ is an ideal. 
	\end{lemma}
	\begin{proof}
		Let $M$ be a closed translation-invariant subspace of $L^{\Phi}(\mathbb{R})$. Define $M^{\perp}:=\{ h \in L^{\Psi}(\mathbb{R}):\langle f,h \rangle =0 \text{ for all }f \in M  \}$. Now take $f \in M$, $g \in L^{\Phi}(\mathbb{R})$. Then an application of the Hahn-Banach theorem shows that $f*g \in M$ if and only if $\int\limits_{\mathbb{R}} (f*g)(x)h(x)dx=0$ for all $h \in M^{\perp}$. It happens since by Fubini's theorem we have
		\begin{align*}
			\int\limits_{\mathbb{R}} (f*g)(x)h(x)dx & =  \int\limits_{\mathbb{R}} \int\limits_{\mathbb{R}} g(y)f(x-y)h(x)dydx\\
			& = \int\limits_{\mathbb{R}}\left(\int\limits_{\mathbb{R}} f(x-y)h(x)dx\right)g(y)dy \\
			& = 0 ~~[\text{since } f \in M \text{ and } h \in M^{\perp}]
		\end{align*}
		thereby finishing the claim.
	\end{proof}
	
	Using the following result from the theory of commutative Banach algebra, we see that every proper closed ideal in $L^{\Phi}(\mathbb{R})$ is contained in a regular maximal ideal.
	\begin{lemma}\cite[p.~85]{loomis1953}\label{loomis_result}
		Let $B$ be a regular semi-simple Banach algebra with the property that the set of elements $f$ such that its Gelfand transform $\hat{f}$ has compact support is dense in $B$. Then every proper closed ideal is included in a regular maximal ideal. 
	\end{lemma}
	
	For this, we first identify the maximal ideal space of $L^{\Phi}(\mathbb{R})$.
	\begin{lemma}
		The maximal ideal space of $L^{\Phi}(\mathbb{R})$ is equivalent to $\widehat{\mathbb{R}}$ and the Gelfand transform can be identified with the Fourier transform.
	\end{lemma}
	\begin{proof}
		Since $L^{\Phi}(\mathbb{R}) \subset L^1(\mathbb{R})$ with $ \|\cdot\|_{1} \leq \|\cdot\|_{L^{\Phi}}$, we have that $\widehat{\mathbb{R}}$ is contained in the maximal ideal space of $L^{\Phi}(\mathbb{R})$.\\
		Indeed, if $\zeta \in \widehat{\mathbb{R}}$, then the Fourier transform of $f$ at $\zeta$ defined by
		$$\hat{f}(\zeta):=\int\limits_{\mathbb{R}}f(x)e^{ix\zeta} dx$$
		defines a multiplicative linear functional on $L^1(\mathbb{R})$. Thus the map $\nu_{\zeta}:~ L^{\Phi}(\mathbb{R}) \rightarrow \mathbb{C}$ given by $\nu_{\zeta}(f) = \mathcal{F}{f}(\zeta)$	is multiplicative and $|\nu_{\zeta}(f)| \leq \|f\|_1 \leq \|f\|_{L^{\Phi}}$ for all $f \in L^{\Phi}(\mathbb{R})$. \\
		
		Now, for the reverse inclusion, let $\nu$ be a multiplicative linear functional on $L^{\Phi}(\mathbb{R})$. By duality, there exists $g \in L^{\Psi}(\mathbb{R})$ such that $\nu(f)=\langle f,g\rangle=\int\limits_{\mathbb{R}} f(x)g(x)dx$ for all $f \in L^{\Phi}(\mathbb{R})$.\\
		Since $\nu$ is multiplicative, for all $f_1,f_2 \in L^{\Phi}(\mathbb{R})$ we have
		$$\int\limits_{\mathbb{R}}\int\limits_{\mathbb{R}} f_1(x)f_2(y)g(x)g(y)dxdy= \nu(f_1)\phi(f_2)=\nu(f_1*f_2)=\int\limits_{\mathbb{R}}\int\limits_{\mathbb{R}} f_1(x)f_2(y)g(x+y)dxdy$$
		which implies that $g(x)g(y)\equiv  g(x+y)$ as elements in $L^{\Psi}(\mathbb{R}^2,\mathbb{C})$ proving that $g(x)g(y)=  g(x+y)$ a.e.
		But the only solution to this is $g(x)=e^{(a+ib)x}$ for $a,b \in \mathbb{R}$. 
		Since $g \in L^{\Psi}(\mathbb{R})$, we have that $a=0$ and thus $\nu=\nu_{b}$.
	\end{proof}
	
	From the above results, we have the following:
	\begin{lemma}\label{proper_maximal}
		Every proper closed ideal of $L^{\Phi}(\mathbb{R})$ is contained in a regular maximal ideal.
	\end{lemma}
	\begin{proof}
		Using Lemma \ref{loomis_result}, it is enough to show that $L^{\Phi}(\mathbb{R})$ is a regular semi-simple Banach algebra such that the set of those functions whose Gelfand transform has compact support is dense in $L^{\Phi}(\mathbb{R})$. But note that the Gelfand transform on $L^{\Phi}(\mathbb{R})$ is equivalent to the Fourier transform, which is unique  on $L^1(\mathbb{R}) \supseteq L^{\Phi}(\mathbb{R})$. Hence $L^{\Phi}(\mathbb{R})$ is semi-simple. Also $L^{\Phi}(\mathbb{R})$ is regular, since the topology on $\widehat{\mathbb{R}}$ given by the Euclidean distance $d(\zeta,\eta)=|\zeta-\eta|$ is equivalent to the weak topology on $\widehat{\mathbb{R}}$ generated by functions in the Gelfand representation of $L^{\Phi}(\mathbb{R})$, which follows simply from the fact $\hat{f}$ is a bounded and continuous function for every $f \in L^{\Phi}(\mathbb{R})$. Indeed, if $C \subset  \widehat{\mathbb{R}} $ is weakly closed and $p \notin C$, then there exists  $0 \neq f \in \mathcal{C}_c^{\infty}(\mathbb{R})$ such that it is supported in $\left(p-\frac{\min \{|p-0|, d(p,C)\}}{2},p+\frac{\min \{|p-0|, d(p,C)\}}{2}\right) \subset \widehat{\mathbb{R}} \backslash C$ and $f(p) \neq 0 $. Thus $g=\mathcal{F}^{-1}[f]$, the inverse Fourier transform of $f$, is in $  L^{\Phi}(\mathbb{R})$ such that $ \hat{g} \equiv 0$ on $C$ and $\hat{g}(p) \neq 0$, proving that $L^{\Phi}(\mathbb{R})$ is a regular Banach algebra.
	\end{proof}

	We observe that a certain behavior of the corresponding Orlicz function around $0$ ensures a characterization of the functions whose all translates are complete like $L^1(\mathbb{R})$. 
	Since $x \prec_{\infty} \Phi_1$ for any Orlicz function $\Phi_1$, so Theorem \ref{embedding_result} ensures that $L^{\Phi}(\mathbb{R}) \subseteq L^1(\mathbb{R})$ if and only if $\lim\limits_{x \to 0} \frac{\Phi(x)}{x}>0$. Thus we have the following.
	
	\begin{theorem}\label{theorem_characterization_completeness_>0}
		Let $\Phi$ be an Orlicz function such that $\Phi \in \Delta_2$. Let $\lim\limits_{x \to 0}\frac{\Phi(x)}{x}>0$. Then all translates of a function $f \in L^{\Phi}(\mathbb{R})$ are complete if and only if its Fourier transform $\hat{f}(\zeta) \neq 0$ for all $\zeta \in \mathbb{R}$. 
	\end{theorem}
	\begin{proof}
		First, suppose that all the translates of $f$ are complete in $L^{\Phi}(\mathbb{R})$. Since $L^{\Phi}(\mathbb{R})$ is dense in $L^1(\mathbb{R})$ with $\|\cdot\|_{L^1} \leq \|\cdot\|_{L^{\Phi}}$, so a density argument imply that all translates of $f$ are also complete in $L^1(\mathbb{R})$. For if let $g \in L^1(\mathbb{R})$. Then there exists $h \in L^{\Phi}(\mathbb{R})$ such that $\|g-h\|_{L^1} < \epsilon/2$. Also, there exists $\tilde{f} \in \mbox{ span} \{ \tau_{\lambda}f\}_{\lambda \in \mathbb{R}}$ such that $\|h-\tilde{f}\| _{L^{\Phi}}< \epsilon /2$. 
		\begin{equation*}
			\|g-\tilde{f}\|_{L^1} \leq \|g-h\|_{L^1}+ \|h-\tilde{f}\| _{L^1} < \epsilon/2 + \epsilon/2 = \epsilon.
		\end{equation*}
		Now since all translates of $f$ are complete in $L^1(\mathbb{R})$, so by Wiener-Tauberian theorem we must have $\hat{f}\neq 0$ everywhere.
		
		Now suppose $f \in L^{\Phi}(\mathbb{R})$ such that $\hat{f}(\zeta)\neq 0$ everywhere. Note that $M_f=\overline{\text{span}\{ \tau_{\lambda}f \}}_{\lambda \in \mathbb{R}}$ is a closed translation-invariant subspace of $L^{\Phi}(\mathbb{R})$, and thus by Lemma \ref{translation_invariant_ideal} is a closed ideal. If $M_f \neq L^{\Phi}(\mathbb{R})$, then by Lemma \ref{proper_maximal}, it is contained in a regular maximal ideal. So there exists $\zeta \in \widehat{\mathbb{R}}$ such that $\hat{f}(\zeta)=0$ which is a contradiction to the choice of $f$. 
	\end{proof}

	Now as a consequence of Theorem \ref{theorem_characterization_completeness_>0}, we identify those discrete $\Lambda \subseteq \mathbb{R}$ such that there exists a function in these Orlicz spaces whose $\Lambda$-translates are complete in $L^{\Phi}(\mathbb{R})$. This characterization is based on the Beurling-Malliavin Density theorem \cite{olevskiiulanovskii16} which shows that $D_{BM}(\Lambda)$ is actually a multiple of the spectral radius $R(\Lambda)$ of $\Lambda$, i.e.,
	$$\pi D_{BM}(\Lambda) = R(\Lambda):= \sup \{ r>0 \text{ such that } E(\Lambda) \text{ is dense in } L^2(-r,r) \},$$
	where $$E(\Lambda):= span \{ e^{i \lambda \cdot} : \lambda \in \Lambda \}$$ is the set of all trigonometric polynomials with frequencies from $\Lambda$. 
	
	\begin{corollary}\label{corollary_discrete_translates}
		Let $\Phi \in \Delta_2$ be an Orlicz function such that $\lim\limits_{x \to 0}\frac{\Phi(x)}{x}>0$ and $ \Lambda \subseteq \mathbb{R}$ be discrete. There exists a function $f \in L^{\Phi}(\mathbb{R})$ whose $\Lambda$-translates are complete exactly when the Beurling-Mallivian density $D_{{BM}}(\Lambda)$ is infinite.
	\end{corollary}
	\begin{proof}

		Suppose $D_{BM}(\Lambda)$ is not finite. By convexity, the Orlicz function $x$ majorizes every Orlicz function around $0$ and by Bruna et.al. \cite{brunaolevskii2006}, there exists a function $f \in L^1(\mathbb{R})$ such that its $\Lambda$-translates are complete. Thus, from Theorem \ref{inclusion_result}, there exists a function in any separable Orlicz space $L^{\Phi}(\mathbb{R})$ such that its $\Lambda$-translates are complete.

		We will show that whenever $D_{BM}(\Lambda)$ is finite, then for any function $f \in L^{\Phi}(\mathbb{R})$ with $\hat{f}\neq 0$ everywhere, there exists some $0 \neq g \in L^{\Psi}(\mathbb{R})$ which annihilates all $\Lambda$-translates of $f$.\\
		Suppose there exists $f \in L^{\Phi}(\mathbb{R})$ such that its $\Lambda$-translates are complete. Then by Theorem \ref{theorem_characterization_completeness_>0}, we must have that  $\hat{f}\neq 0$ everywhere on $\mathbb{R}$. Also, by the Beurling-Malliavin Density Theorem, there is an $a>0$ such that $E(\Lambda)$ is not dense in $L^2(-a,a)$, implying the existence of a $0 \neq h \in L^2(\mathbb{R}) \cap \mathcal{C}_c^{\infty}(\mathbb{R})$ such that $h(\lambda)=0$ for all $\lambda \in \Lambda$ and $\mbox{supp }(\hat{h}) \subseteq (-a,a)$. Thus we can define a function $0 \neq G \in \mathcal{C}_c^{\infty}(\mathbb{R})$ by
		\begin{equation*}
			G(\xi) = \begin{cases} 
				\frac{\hat{h}(\xi)}{\hat{f}(\xi)} & \text{ for }\xi \in supp(\hat{h}), \\
				0 & \text{ otherwise} .
			\end{cases}
		\end{equation*}
		Thus $g=\check{G} \in \mathcal{S}(\mathbb{R}) \subseteq L^{\Psi}(\mathbb{R})$ such that for any $\lambda \in \Lambda$, we have
		\begin{equation*}
			\langle \tau_{\lambda}f ,g \rangle\ = \int\limits_{\mathbb{R}}\hat{g}(\xi) \hat{f}(\xi) e^{i \lambda \xi} d \xi = \int\limits_{supp(\hat{h})}\hat{h}(\xi)e^{i \lambda \xi} d \xi = \sqrt{2 \pi} h(\lambda) = 0.
		\end{equation*}
		This shows that $\Lambda$-translates of $f$ are not complete in $L^{\Phi}(\mathbb{R})$, a contradiction.
	\end{proof}
	
	In particular, for uniformly discrete sets  $\Lambda$, we have $D_{BM}(\Lambda) < \infty$. So there does not exist any function such that its $\Lambda$-translates are complete in $L^{\Phi}(\mathbb{R})$ when the Orlicz function $\Phi \in \Delta_2$ satisfy $\lim\limits_{x \to 0}\frac{\Phi(x)}{x}>0$.\\
	
	\begin{remark}
		In contrast to the above discussion, when the Orlicz function $\Phi \in \Delta_2$ is such that $\lim\limits_{x \to 0}\frac{\Phi(x)}{x}=0$, then there does exist uniformly discrete set $\Lambda$ so that $L^{\Phi}(\mathbb{R})$ has a function whose $\Lambda$-translates are complete. In particular, we can take $\Lambda$ to be a very small perturbation of $\mathbb{Z}$, i.e., 
		\begin{equation*}
			\begin{split}
				\Lambda = \{ n+r_n\}_{n \in \mathbb{Z}}, \hspace{1.8in}\\
				0 < |r_n| \leq \gamma^{|n|} \text{ for all } n \in \mathbb{Z} \text{ and some } 0 < \gamma < 1. 
			\end{split}
		\end{equation*}
		It holds by \cite[Proposition 3.1]{olevskiiulanovskii18} since by the given conditions, the space $L^{\Phi}(\mathbb{R})$ is a separable symmetric space which is not a subspace of $L^1(\mathbb{R})$. \\
	\end{remark}

	Lastly, we show that whenever the Orlicz function $\Phi \in \Delta_2$ satisfies $\lim\limits_{x \to 0}\frac{\Phi(x)}{x}=0$, then all translates of any simple step function of the form $\chi_{[0,h]}$ for some $h>0$ are complete in the Orlicz space. The proof lies in the construction given by Agnew in \cite{agnew1945I} and a simple observation regarding the norm of a characteristic function in the Orlicz space $L^{\Phi}(\mathbb{R})$.

	Agnew \cite{agnew1945I} proved that all translates of a simple step function are dense in $L^p(\mathbb{R})$ for any $p>1$. The main idea in both the papers was the approximation of a dense class of functions via the respective choice in such a way that approximation error is exactly a step function of the form $\frac{\chi_{[0,nd]}}{n}$. Now as $p>1$, the norm $\| \frac{\chi_{[0,nd]}}{n}\|_p$ can be made arbitrarily small by a suitable large choice of $n$. We observed that for a general Orlicz function $\Phi$, the property that $\lim\limits_{x \to 0}\frac{\Phi(x)}{x}=0$ is sufficient to ensure that the same thing holds if we replace $\|\cdot\|_p$-norm by $\|\cdot \|_{L^{\Phi}}$-norm.
	
	\begin{lemma}\label{main_lemma_step_tent}
		Let $\Phi$ be an Orlicz function such that $\lim\limits_{x \to 0}\frac{\Phi(x)}{x}=0$. For every $\epsilon, d >0$, there exists $n_{\epsilon,d} \in \mathbb{N}$ such that 
		\begin{equation*}
			\left\| \frac{\chi_{[0,nd)}}{n} \right\|_{L^{\Phi}} < \epsilon ~~~~\mbox{ for all } n \geq n_{\epsilon,d}.
		\end{equation*}
	\end{lemma}
	\begin{proof}
		Since $\lim\limits_{x \to 0}\frac{\Phi(x)}{x}=0$, so there exists $\delta>0$ such that 
		\begin{equation}\label{EQ:1}
			\frac{\Phi(x)}{x}< \frac{\epsilon}{d} \mbox{ for all }0<  x < \delta.
		\end{equation} 
		Also, we know that $\frac{1}{n} \to 0$, so we can choose $n\in \mathbb{N}$ such that 
		\begin{equation*}
			\frac{1}{n}< \epsilon \delta \mbox{ for all } n \geq n_{\epsilon,d}.
		\end{equation*}
		Thus for all $n \geq n_{\epsilon,d}$, we have
		\begin{equation*}
			\frac{\Phi\left(\frac{1}{n \epsilon}\right)}{\frac{1}{n \epsilon}}< \frac{\epsilon}{d},
		\end{equation*}
		implying that
		\begin{equation*}
			\frac{1}{n \epsilon} < \Phi^{-1} \left(\frac{1}{nd}\right),
		\end{equation*}
		which in turn gives that
		\begin{equation*}
			\left\| \frac{\chi_{[0,nd)}}{n} \right\|_{L^{\Phi}}  = \frac{1}{n} \left[ \Phi^{-1}\left(\frac{1}{nd}\right) \right]^{-1} < \epsilon.
		\end{equation*}
		This completes the proof of the lemma.
	\end{proof}

	Now by taking the same construction as in \cite{agnew1945I} and combining with Lemma \ref{main_lemma_step_tent}, we get the desired completeness property:
	\begin{theorem}\label{theorem_completeness_step_tent}
		Let $\Phi$ be an Orlicz function such that $\Phi \in \Delta_2$ and $\lim\limits_{x \to 0}\frac{\Phi(x)}{x}=0$. Then all translates of a simple step function are complete in $L^{\Phi}(\mathbb{R})$. 
	\end{theorem}
	\begin{proof}
		Since $\Phi \in \Delta_2$, so step functions are dense in $L^{\Phi}(\mathbb{R})$. Therefore, in order to show that all translates of some function are complete, it is enough to prove that we can arbitrarily approximate any characteristic function $g:=\chi_{[0,a]}$ for $a>0$.
		
		Let $f$ be a simple step function of width $b >0$, say $f=\chi_{[0,b]}$. Clearly, the function $\chi_{[0,nb]} \in \mbox{ span}\{\tau_{\lambda}f\}_{\lambda \in \mathbb{R}}$ for every $n \in \mathbb{N}$. Choose $m$ to be the least positive integer for which $mb>a$.
		
		Let $\epsilon>0$ and choose $n=n_{\epsilon,2mb}$ from Lemma \ref{main_lemma_step_tent}.
		
		If we define
		\begin{equation*}
			g_{\epsilon}(x)= \sum\limits_{k=0}^{n-1} \left[ \left(1-\frac{2k+1}{2n+1}\right)\chi_{[0,mb]}(x-kmb)-\left(1-\frac{2k+2}{2n+1}\right)\chi_{[0,mb]}(x-kmb-a)  \right],
		\end{equation*}
		then $g_{\epsilon} \in \mbox{ span}\{\tau_{\lambda}f\}_{\lambda \in \mathbb{R}}$ and
		\begin{equation*}
			|g-g_{\epsilon}|= 	\frac{1}{2n+1}\chi_{[0,nmb+a]}.
		\end{equation*}
		Therefore by Lemma \ref{main_lemma_step_tent} we have
		\begin{equation*}
			\|g-g_{\epsilon}\|_{L^{\Phi}}<  \epsilon .
		\end{equation*}
		This completes the proof of the theorem.
	\end{proof}

	In particular, for such $\Phi$, there exists $f \in L^{\Phi}(\mathbb{R})$ such that its Fourier transform $\hat{f}(\zeta) =0$ for some $\zeta \in \mathbb{R}$ and still its all translates complete in the Orlicz space $L^{\Phi}(\mathbb{R})$.\\
	
	\begin{remark}
		Lemma \ref{main_lemma_step_tent} can also be used to show completeness of all translates of any tent function of the form 
		\begin{equation*}
			f(x)=
			\begin{cases}
				a(b-|x|) & |x|\leq b \\
				0 & |x|>b \\
			\end{cases}
		\end{equation*}
		for some $a,b>0$ in the Orlicz space $L^{\Phi}(\mathbb{R})$ whenever $\lim\limits_{x \to 0}\frac{\Phi(x)}{x}=0$ with $\Phi \in \Delta_2$. It can be done similarly to the above proof, by using the construction in \cite{agnew1945II} where it is proved that all translates of any tent function are complete in $L^p(\mathbb{R})$ for any $p>1$. \\
	\end{remark}

	\section*{Acknowledgments}
	
	The second author acknowledges the Science and Engineering Research Board, Government of India, for the financial support through project no. MTR/2022/000383.\\

	\bibliographystyle{abbrv}

\end{document}